\documentclass[twoside,11pt]{article}
\usepackage{isipta}

\usepackage[utf8]{inputenc}
\inputencoding{utf8} 

\usepackage{amsmath}
\usepackage{array}
\usepackage{standalone}
\usepackage{lscape}
\usepackage{epsfig}
\usepackage[all]{xy}
\usepackage{enumerate}
\usepackage{xcolor}
\usepackage{graphicx}
\usepackage{tikz}
\usetikzlibrary{matrix}
\usepackage{hyperref}

\hypersetup{colorlinks=true,linkcolor=blue,citecolor=magenta}

\usetikzlibrary{positioning,calc,fit,shapes.geometric,patterns}
\pgfdeclarelayer{background}
\pgfdeclarelayer{foreground}
\pgfsetlayers{background,main,foreground}
\tikzstyle{vec}=[circle,inner sep=1pt,outer sep=-1pt,fill]
\tikzstyle{border}=[thick]
\tikzstyle{favborder}=[border,dotted]
\tikzstyle{exclborder}=[border,dashed]

\usepackage{todonotes}

\newcommand{\BibTeX}{\textsc{B\kern-0.1emi\kern-0.017emb}\kern-0.15em\TeX}

\newcommand {\reals}{\mathbb{R}}

\newcommand {\desir}{\mathcal{K}}

\newcommand {\M}{\mathsf{Max}}

\newcommand {\posi}{\operatorname{posi}}

\newcommand {\GS}{\mathsf{GS}}

\ShortHeadings{A polarity theory for  sets of desirable gambles}{Benavoli et al.} 

\begin{document}

\title{A polarity theory for  sets of desirable gambles}

\author{\name Alessio Benavoli \email alessio@idsia.ch\\
\name Alessandro Facchini \email alessandro.facchini@idsia.ch\\
\name Marco Zaffalon \email zaffalon@idsia.ch\\
\addr Istituto Dalle Molle di Studi Sull'Intelligenza Artificiale (IDSIA),
Lugano (Swizterland) \AND
\name Jos\'e Vicente-P\'erez \email jose.vicente@ua.es\\
\addr Departamento de Fundamentos del An\'alisis Econ\'omico, 
Universidad de Alicante (Spain)}

\maketitle

\begin{abstract}\noindent 
Coherent sets of almost desirable gambles and credal sets are known to be equivalent models. That is, there exists a bijection between the two collections of sets preserving the usual operations, e.g. conditioning. Such a correspondence is based on the polarity theory for closed convex cones. Learning from this simple observation, in this paper we introduce a new (lexicographic) polarity theory for general convex cones and then we apply it in order to establish an analogous correspondence between coherent sets of desirable gambles and convex sets of lexicographic probabilities.
\end{abstract}



\section{Introduction}
\label{sec:1}

\Citet{finetti1937} established a foundation of probability theory based on the notion  of ``coherence'' (self-consistency).  
The idea was that a subject is considered rational if she chooses her odds so that there is no bet that leads her to a sure loss (no Dutch books are possible).
In this way, since numerically odds are the inverse of probabilities,  de Finetti's approach provides a justification of Kolmogorov's axioms of probability as a rationality criterion on a gambling system.

Later, building on de Finetti's betting setup, \citet{williams1975} and then \citet{Walley91} have shown that it is possible to justify probability in a way that is even simpler, more general and elegant. The basic idea is that  an agent's knowledge about the outcome of an experiment to be performed (e.g. tossing a coin) is provided by her set of \emph{desirable} gambles, that is the set of gambles she is ready to accept. A gamble is modelled as a real-valued function $g$ on the set $\Omega$ of outcomes of the experiment. Hence by accepting a gamble $g$, an agent commits herself to 
receive $g(\omega)$ \emph{utiles} in case  the experiment is performed and the outcome of the experiment eventually happens 
to be the event $\omega \in \Omega$. 
Among all the sets of desirable gambles, we are able to find those satisfying some properties, and called \textit{coherent} sets of desirable gambles, as they represent rational choices. Mathematically, those properties boil down to ask for a coherent set of desirable gambles to be a convex cone without the origin that contains all positive gambles, and thus avoids the negative ones (avoids partial loss).  In spite of its simplicity, the \emph{theory of desirable gambles} encompasses not only the Bayesian theory of probability 
but also other important mathematical models like upper and lower previsions or (credal) sets of probabilities. 

An important variant of the traditional theory of probability is the probabilistic model of lexicographic probabilities \citep{blume1991lexicographic}, that is a sequence of standard probability measures. Developed to deal with the problem of conditioning on events of measure 0, it shares several features not only with models such as conditional probabilities or non-standard probabilities, but also with the theory of desirable gambles \citep[see, e.g.,][]{seidenfeld1990decisions,seidenfeld2000remarks,cozman2015,van2017lexico}. In particular
\citet{cozman2015} notices that  (conditional) sets of desirable gambles expressed  via preference relations can be represented by sets of (conditional) lexicographic probabilities. This fact leads us to wonder whether, analogously to the case of sets of almost desirable gambles and sets of probabilities, a stronger, more fundamental correspondence  exists between sets of desirable gambles and sets of lexicographic probabilities.
 
The goal of the present paper is to show that this is the case. That is, we verify that (conditional) sets of lexicographic probabilities and (conditional) sets of desirable gambles  are isomorphic structures. In doing so,  we provide a duality transformation (via orthogonal matrices) that allows  us to go from a coherent set of desirable gambles to an equivalent set of lexicographic probabilities and vice versa.  This transformation is an important contribution to  uncertainty modelling  because having access to  dual models of uncertainty enables greater freedom of expression. In particular, we believe that the possibility of transferring through duality constructions from one theory to the other can be used to better understand issues related to lexicographic probabilities, such as defining independence.


\section{Preliminaries}
\label{sec:2}

We start by introducing the necessary notation and basic definitions to be used later. Assume that the set of outcomes of an experiment is finite, say $\Omega=\{\omega_1,\ldots,\omega_n\}$, and that there is an unknown true value in $\Omega$. A gamble $g$ on $\Omega$ is a mapping $g:\Omega \rightarrow \mathbb{R}$, and so $g(\omega)$ represents the reward the gambler would obtain if $\omega$ is the true unknown value. As the cardinality of $\Omega$ is $n$ (a natural number), every gamble $g$ on $\Omega$ can be thought as a point in the Euclidean space $\mathbb{R}^{n}$, and hence write $g=(g_1,\ldots,g_n)$ with $g_i\in\mathbb{R}$ for every $i\in N:=\{1,\ldots,n\}$. In line with the tradition within the imprecise probability community, the set of all gambles defined on $\Omega$ is denoted by $\mathcal{L}(\Omega)$, although at times we simply write $\mathbb{R}^n$. 

The elements of $\mathbb{R}^n$ will be considered column vectors and the symbol $^{\top}$ will mean transpose. We denote by $0_n$ ($-1_n$, respectively) the vector whose components are all equal to $0$ ($-1$, respectively). The vectors $e^1,\ldots,e^n$ stand for the canonical basis of $\mathbb{R}^{n}$, that is, $e^i$ is the vector of zeros with a one in the $i$-th position, for all $i\in N$. Given $g,f\in\mathbb{R}^{n}$, the standard inner product of $g$ and $f$ is $\langle g,f\rangle := g^{\top}f$ and the Euclidean norm of $g$ is $\|g\|:=\sqrt{\langle g,g\rangle}$. For any subset $C \subset \reals^n$, we denote by $\posi(C)$ the set of all positive linear combinations of gambles in $C$, that is, $ \posi(C):= \{ \sum_{j=1}^{m}{\lambda_j g^j} : g^j \in C, \lambda_j >0, m \in\mathbb{N}\}$. We say that $g$ is \emph{less than or equal to} $f$ (in short, $g \leq f$) whenever $g_i \leq f_i$ for all $i \in N$, and we will write $g < f$ whenever $g \leq f$ and $g\neq f$.  The set of non-negative gambles is $\mathbb{R}^{n}_{+} := \{g\in\mathbb{R}^{n} : g \geq 0_n \}$.
Furthermore, $g$ is said to be \emph{lexicographically less than} $f$ (in short, $g <_{L} f$) if $g\neq f$ and $g_{k} < f_{k}$ for $k:=\min \left\{ i\in N : g_{i}\neq f_{i}\right\} $. We also write $g\leq_{L} f $ if either $g<_{L}f$ or $g=f$. \smallskip

The following properties for a subset $\desir \subset \reals^n$ will be needed below. \smallskip

\noindent \textbf{A1.} If $g\in \desir$ and $f\in \desir$, then $g+f \in \desir$ (addition). \smallskip

\noindent \textbf{A2.} If $g\in \desir$ and $\lambda>0$, then $\lambda g\in \desir$ (positive homogeneity). \smallskip

\noindent \textbf{A3.} If $g > 0_n$, then $g\in \desir$ (accepting partial gain). \smallskip

\noindent \textbf{A4.} $0_n\notin \desir$ (avoiding status quo). \smallskip

\noindent \textbf{A5.} If $g < 0_n $, then $g\notin \desir$ (avoiding partial loss). \smallskip

\noindent \textbf{A6.} $-1_n \notin \desir$ (avoiding sure loss). \smallskip

\noindent \textbf{A7.} If $g +f \in \desir$ for all $f>0_n$, then $g\in \desir$ (closure). \smallskip

\noindent \textbf{A8.} $0_n\in \desir$ (accepting status quo).

\begin{definition}
	A subset $\desir \subset \reals^n$ is said to be a \emph{coherent set of}
	\begin{description}
	\item[$\bullet$] \emph{desirable gambles}
	if it satisfies properties A1, A2, A3, A4;
	\item[$\bullet$] \emph{almost desirable gambles} if it satisfies properties A1, A2, A3, A6, A7.
	\end{description}
\end{definition}
Thus, it easily follows that a coherent set of desirable gambles also satisfies properties A5 and A6, and a coherent set of almost desirable gambles also satisfies property A8. By definition, one has that the elements of $\mathbb{D}_n$, the family of all coherent sets of desirable gambles on $\Omega$, are convex cones in $\mathbb{R}^{n}$ omitting their apex (the origin), whereas the elements of $\mathbb{A}_n$, the family of all coherent sets of almost desirable gambles on $\Omega$, are closed convex cones (containing the origin) in $\mathbb{R}^{n}$. However, not every convex cone omitting its apex (closed convex cone, respectively) belongs to $\mathbb{D}_n$ ($\mathbb{A}_n$, respectively).

A crucial tool for duality within the framework of Convex Analysis is the polarity operator. Given a convex cone $K\subset \mathbb{R}^{n}$, the \emph{(positive) polar of $K$} is defined to be 
\begin{equation*}
K^{\circ}:=\{v\in \mathbb{R}^{n} : \langle v,g\rangle \geq 0 \text{ for all } g\in K\}. 
\end{equation*}
Note that $K^{\circ}$ is a closed convex cone (containing the origin). Furthermore, one has $K^{\circ\circ} = \operatorname{cl} K$ \cite[see][]{R70}, 
 and for closed convex cones $K_1,K_2\subset\mathbb{R}^n$, one has  $K_1\subset K_2$ if and only if $K_2^{\circ} \subset K_1^{\circ}$. 

Let $m\in\mathbb{N}$ with $m\leq n$. The symbol $\mathbb{M}_{m,n}$ denotes the space of real matrices with $m$ rows and $n$ columns, whereas 
$\mathbb{O}_{m,n}$ denotes the subset of matrices in $\mathbb{M}_{m,n}$ with orthonormal rows, that is, those matrices $A$ satisfying $A A^{\top} = I$ (where $I$ is the identity matrix of appropriate order). For $A\in\mathbb{M}_{m,n}$ we denote by $a_{ij}$ the element of $A$ in row $i$ and column $j$, the $i$-th row of $A$ is denoted by $a_{i \cdot}$, whereas its $j$-th column is denoted by $a_{\cdot j}$. 
Given $A\in\mathbb{M}_{n,n}$, we write $A \geq_L (>_L) \ 0_n$ \citep[in the sense of][]{JEML84} if each column of $A$ satisfies $a_{\cdot j} \geq_L (>_L) \ 0_n$ for all $j\in N$.

A \emph{probability mass function} over $\Omega$ is any vector belonging to the set 
$$  \mathbb{P}_n := \left\{ p\in\mathbb{R}^{n} : 0\leq p_i \leq 1, \sum_{i\in N}p_i=1\right\} . $$ 
Any closed convex subset of $\mathbb{P}_n$ is called a \emph{credal set}. We shall denote by $\mathbb{C}_n$ the family of all credal sets within $\mathbb{P}_n$. A \emph{lexicographic probability} over $\Omega$ is a sequence $\{p^j\}_{j=1}^{m}$ with $p^j \in \mathbb{P}_n$. We usually identify lexicographic probabilities over $\Omega$ with \emph{stochastic matrices}, that is,
$$ \mathbb{S}_{m,n} := \left\{ P \in\mathbb{M}_{m,n} : p_{i\cdot} \in  \mathbb{P}_n\ \text{ for all } i=1,\ldots,m\right\}.   $$
We shall denote by $\mathbb{T}_{m,n}$ the subset of $\mathbb{S}_{m,n}$ containing all the full-rank stochastic matrices. 
%
%
%
%

\section{Almost desirability and probability}
\label{sec:3}


\Citet{Walley91} showed that there is a one-to-one correspondence between coherent sets of almost desirable gambles and credal sets, say $\mathbf{C}:\mathbb{A}_n \rightarrow \mathbb{C}_n$. Moreover, it is often claimed that this correspondence actually shows that  the theory of almost desirable gambles and the theory of credal sets
are equivalent.
In this section, we first recall the bijection $\mathbf{C}$ which is based on the polarity theory for closed convex cones \citep{R70}. Second, by using the point of view of model theory  \cite[see e.g.][]{hodges1997shorter}, 
we explain how one has to understand the claim that the theory of almost desirable gambles and the theory of credal sets
are equivalent. Finally, we prove the claim. 

\subsection{Polarity for almost desirability}

The underlying tool for getting the aforementioned bijection is the \emph{classical separation theorem for closed convex sets}: if $\desir \subset \mathbb{R}^{n}$ is a nonempty closed convex cone, then for every $\overline{g}\notin \desir$ there exists $v\in\mathbb{R}^{n}$ (non-null) such that $\langle v,g\rangle \geq 0 > \langle v,\overline{g}\rangle$ for all $g\in\desir$. Thus, every closed convex cone  $\desir \subset \mathbb{R}^{n}$ can be written as $\desir = \{g\in\mathbb{R}^{n} : \langle v^t,g\rangle \geq 0, t\in T \}$ for certain $v^t\in\mathbb{R}^{n}$ and $T$ an arbitrary index set. In such a case, a well-known result in Convex Analysis \citep[see][]{R70} states that $\desir^{\circ}$ coincides with the closure of the conic convex hull of the $\{v^t, t\in T\}$. In particular, if $\desir = \{g\in\mathbb{R}^{n} : \langle v,g\rangle \geq 0 \}$ with $v\in\mathbb{R}^{n}$, then $\desir^{\circ} = \mathbb{R}_{+}v = \{\lambda v : \lambda \geq 0\}$.  Concerning the geometry of coherent sets of almost desirable gambles, any set $\desir \in \mathbb{A}_{n}$ is characterised as a closed convex cone containing the set $\mathbb{R}^{n}_{+}$ (or equivalently, containing all indicator gambles). Thus, as a particular case, since any $\desir \in\mathbb{A}_n$ is a closed convex cone containing $\{e^1,\ldots,e^n\}$, the following proposition holds. 

\begin{proposition}
\label{separalmdesir}
Let $\desir \in \mathbb{A}_n$ and $\overline{g}\notin \desir$. Then, there exists $v\in\mathbb{R}^{n}$ with $v > 0_n $ and $\|v\|=1$ such that $\langle v,g\rangle \geq 0_n > \langle v,\overline{g}\rangle$ for all $g\in\desir$.	
\end{proposition}

\begin{corollary}
	\label{separalmdesircor}
For every $\desir \in\mathbb{A}_n$, there exist an index set $T$ and vectors $v^t > 0_n$ with $\|v^t\|=1$ for all $t\in T$ such that $\desir = \{g\in\mathbb{R}^{n} : \langle v^t, g\rangle  \geq 0, t\in T \}$.
\end{corollary}

Recall that a set $\desir \in \mathbb{A}_n$ is said to be \emph{maximal} if there is no other element $\desir' \in \mathbb{A}_n$ such that $\desir \subsetneq \desir'$. Thus, we have that the maximal elements in $\mathbb{A}_n$ are the closed halfspaces containing the origin in the boundary and determined by vectors with non-negative components and norm 1. Hence, if we denote by $\M(\mathbb{A}_n)$ the set of all maximal elements in $\mathbb{A}_n$, given $\desir\in\mathbb{A}_n$ one has 
\begin{equation}
\label{eq:almostmax}
\desir \in \M(\mathbb{A}_n) \ \Longleftrightarrow \ \exists\,v > 0_n, \|v\|=1\text{ (unique) such that }\desir = \{g\in\mathbb{R}^{n} : \langle v,g\rangle \geq 0 \} . 
\end{equation}
This means that there is a one-to-one correspondence between maximal coherent sets of almost desirables gambles and non-negative vectors with norm 1. Since a bijection between the set of non-negative vectors with norm 1 and $\mathbb{P}_n$ exists, then there is a one-to-one correspondence between maximal coherent sets of almost desirables gambles and probability mass functions over $\Omega$.  
Furthermore, as a consequence of Proposition \ref{separalmdesir}, for any $\desir\in\mathbb{A}_n $ one can write
\begin{equation*}
\desir = \bigcap \{ \desir' \in \M (\mathbb{A}_n) : \desir \subset \desir'\}.
\end{equation*}
The above equality and the one in \eqref{eq:almostmax} imply a reformulation of Proposition \ref{separalmdesir}: if $\desir \in \mathbb{A}_n$ and $g \notin \desir$, then there exists $\desir' \in \M (\mathbb{A}_n)$ such that $\desir \subset \desir'$ and $g \notin \desir'$.

Next we define the function $\mathbf{C}:\mathbb{A}_n \rightarrow \mathbb{C}_n$ which maps coherent sets of almost desirable gambles into credal sets and it is the key for the equivalence of both theories. For a coherent set of almost desirable gambles $\desir\in\mathbb{A}_n$, we associate the credal set
\begin{equation}
\label{mapc}
\mathbf{C}(\mathcal{K}) := \mathcal{K}^{\circ} \cap \mathbb{P}_{n}.
\end{equation}
Observe that if $\desir\in \M(\mathbb{A}_n)$ is determined by $v$ as in \eqref{eq:almostmax}, then $\mathbf{C}(\mathcal{K}) = (\sum_{i\in N}v_i)^{-1}v$.


\begin{theorem}
\label{theo:credal}
The mapping $\mathbf{C}:\mathbb{A}_n \rightarrow \mathbb{C}_n$ defined in \eqref{mapc} is a bijection whose inverse is given by 
$ \mathbf{C}^{-1}(\mathcal{P}) := \mathcal{P}^{\circ} $
for every credal set $\mathcal{P} \in \mathbb{C}_n$.
\end{theorem}
\begin{proof}
First, it is easy to see that, for any $\desir\in\mathbb{A}_n$, the set $\mathbf{C}(\mathcal{K})$ is a credal set. Since $\mathbb{R}^{n}_+ \subset \desir$, one has  $\desir^{\circ} \subset (\mathbb{R}^{n}_+)^{\circ} = \mathbb{R}^{n}_+$. Moreover, $\desir^{\circ}$ does not reduce to $0_n$ (this fact just happens whenever $\desir=\mathbb{R}^{n}$, which does not belong to $\mathbb{A}_n$ indeed) and so, $\desir^{\circ}$ contains non-null non-negative vectors, and particularly, at least one vector with the sum of its components equal to $1$ (up to normalisation). Thus, the set $\mathcal{K}^{\circ} \cap \mathbb{P}_{n} \subset \mathbb{P}_{n} $ is nonempty. Moreover, since both $\mathcal{K}^{\circ}$ and $\mathbb{P}_{n}$ are closed convex sets and closedness and convexity are preserved under intersection, then $\mathbf{C}(\mathcal{K}) \in \mathbb{C}_n$. 


We have shown that the mapping $\mathbf{C}$ is well-defined, associating a credal set to each coherent set of almost desirable gambles. Next, we verify that $\mathbf{C}$ is a bijection, that is, for any credal set $\mathcal{P} \in\mathbb{C}_n$, there exists a unique $\desir\in\mathbb{A}_n$ such that $\mathbf{C}(\mathcal{K}) = \mathcal{P}$.

Given a credal set $\mathcal{P} \in\mathbb{C}_n$, it follows that $\mathbb{R}_{+}\mathcal{P}$ is a closed convex cone contained in $\mathbb{R}^{n}_{+}$. Thus, by taking polars one has $\mathbb{R}^{n}_+ = (\mathbb{R}^{n}_+)^{\circ} \subset (\mathbb{R}_{+}\mathcal{P})^{\circ} = \mathcal{P}^{\circ}$ and so, $\mathbf{C}^{-1}(\mathcal{P}) \in \mathbb{A}_{n}$ as $\mathcal{P}^{\circ}$ is a closed convex cone containing $\mathbb{R}^{n}_{+}$. Indeed, $\mathbf{C}^{-1}(\mathcal{P}) \in \mathbb{A}_{n}$ is the unique coherent set of almost desirable gambles satisfying $\mathbf{C}(\mathbf{C}^{-1}(\mathcal{P})) = \mathcal{P}$. Furthermore, for any $\desir\in\mathbb{A}_n$ one has $\mathbf{C}^{-1}(\mathbf{C}(\desir)) = \desir$.  
\end{proof}

\subsection{Theories as structures, and equivalence as isomorphism}

The fact that $\mathbf{C}$ establishes a bijection between coherent sets of almost desirable gambles and credal sets is clearly not enough for claiming that the two theories are equivalent. We also need to verify that such a mapping preserves all considered operations (like conditioning and marginalisation) and relations (like independence). In other words, we have to verify that it  is an isomorphism, once the two theories, from the point of view of model theory, are formulated as structures on the same signature. To illustrate this point, let us assume that we are only interested in conditioning. From a model-theoretic point of view, this means that we are considering a signature consisting of only a unary functional symbol. 
The next steps are thence the following: (i) we have to state how the considered operation is defined over coherent sets of almost desirable gambles and over credal sets (in model-theoretic terms, we have to specify how the elements of the signature -- in this case its unique element -- must be interpreted in both cases), and then (ii) we have to show that the map $\mathbf{C}$ preserves the considered operation (in model-theoretic terms, we have to verify that the map is a homomorphism).

 Here below we thence recall the definition of  this operation within the theory of almost desirable gambles as given in \cite{de2012exchangeability}, a slightly different but completely equivalent  version as the one in \cite{Walley91}. To this aim, given a subset $\Pi\subsetneq \Omega$ of cardinality $m<n$, we shall denote by $\Pi^c$ the set of outcomes which are not in $\Pi$, that is, $\Pi^c := \Omega\backslash \Pi$. For a gamble $g\in\reals^m$ we define the gamble $(g \lceil_{\Pi^c}) \in \reals^n$ as $(g\lceil_{\Pi^c})(\omega) := g(\omega)$ if $\omega\in \Pi$ and $(g\lceil_{\Pi^c})(\omega) := 0$ if $\omega \in \Pi^c$. 

\begin{definition}
	\label{condition_set}
	Let $\desir \subset \reals^n$. The conditioned set of $\desir$ with respect to $ \Pi$ is the set 
	$$ (\desir \rfloor_\Pi) := \{g \in \reals^m : (g \lceil_{\Pi^c}) \in \desir \}.$$
\end{definition}

Notice that conditioning 
does not necessarily preserve coherent sets of almost desirable gambles 
(see \citet[Section 4]{zaffalon2010e} for a thorough discussion on this point). As an example, consider the sets $\Omega=\{1,2\}$, $\Pi=\{2\}$ and $\desir=\{ g \in \mathbb{R}^2 : g_1 \geq 0\}$. Whereas $\desir \in \mathbb{A}_2$, it holds that $(\desir \rfloor_\Pi) = \mathbb{R} \notin \mathbb{A}_1$. 

For a probability mass function $p$ over $\Omega$, let $p(\cdot | \Pi)$ denote the usual conditioning of $p$ with respect to $\Pi \subset \Omega$. 
Hence, 
if $\mathcal{P}\subset \mathbb{P}_{n}$ is a credal set over $\Omega$,  the conditioning of $\mathcal{P}$ on $\Pi $ is the projection on $\Pi$ of all $p(\cdot | \Pi) \in \mathbb{P}_n$, with $p \in \mathcal{P}$; that is
$ 	(\mathcal{P}\rfloor_\Pi):=\{ p \in \mathbb{P}_m : \exists\, q \in  \mathcal{P} \text{ such that } (p \lceil_{\Pi^c}) = q(\cdot | \Pi) \}$. Notice that this definition is completely equivalent  as the usual definition of conditioning for credal sets as given in  \cite{couso2011sets}.

We can thence formulate the missing property for the mapping $\mathbf{C}$ to be called an isomorphism, and thus to be claimed to show the equivalence between the two theories (when the considered operation is conditioning only). 

\begin{theorem}
	Let $\mathcal{K}  \in \mathbb{A}_n$ and $\Pi \subset \Omega$. The following statements hold:  
	\begin{enumerate}
		\item[$(i)$] $(\desir\rfloor_\Pi)\in \mathbb{A}_m$ if and only if $(\mathbf{C}(\mathcal{K})\rfloor_\Pi) \in \mathbb{C}_m$.
		\item[$(ii)$]  If $(\desir\rfloor_\Pi) \in \mathbb{A}_m$, then $\mathbf{C}(\desir\rfloor_\Pi) = (\mathbf{C}(\mathcal{K})\rfloor_\Pi)$.
	\end{enumerate}
\end{theorem}
\begin{proof} It is enough to prove both claims for $\mathcal{K}  \in \M(\mathbb{A}_n)$. Let $\{p\}= \mathbf{C}(\mathcal{K} ) \in \mathbb{C}_n$. With $i_\Pi$ we should denote the indicator gamble on $\Pi$. Since $\langle p, i_\Pi f \rangle= \langle  i_\Pi p, f \rangle$ and Theorem \ref{theo:credal}, the following holds:
\begin{equation}
	\label{cond:dualthm}
	 (\desir \rfloor_\Pi) = 
	  \{ g \in \reals^m : \langle i_\Pi p, f \rangle \geq 0, \text{ for } f \in \reals^n \text{ such that }i_\Pi f = g \lceil_{\Pi^c} \}.
	\end{equation}
Hence, for both points we conclude by applying Theorem \ref{theo:credal} to Equation \ref{cond:dualthm}.
\end{proof}


\section{Desirability and lexicographic probabilities}
\label{sec:4}

As discussed by \citet{cozman2015}, coherent sets of desirable gambles and lexicographic probabilities seem to share several properties. We wonder whether these two models are somehow equivalent, that is, if there is a one-to-one correspondence $\mathbf{G} : \mathbb{D}_n\rightarrow\mathbb{G}_n$ between coherent sets of desirable gambles and certain sets (to be defined later) of lexicographical probabilities, similar to the one existing for credal sets and coherent sets of almost desirable gambles described in  Section \ref{sec:3}. 

\subsection{Polarity for desirability}

As done in  Section \ref{sec:3}, the following (lexicographic) separation theorem for convex sets will be  now the key result for getting the aforementioned equivalence. 

\begin{theorem}[\cite{martinez1983exact}]
\label{thm:lexicosep}
Let $G\subset\mathbb{R}^{n}$ be a nonempty convex set and $\overline{g} \notin G$. Then, there exists $A\in \mathbb{M}_{n,n}$ and $b \in \reals^n$ such that 
	$A g >_L b \geq_L A \overline{g}$ \,for all $g \in G$.
\end{theorem}

The matrix $A$ in the above theorem can be assumed to be full-rank, or even orthonormal. Consequently, every convex set $G \subset \mathbb{R}^{n}$ can be written as $G = \{g\in\mathbb{R}^{n} : A^t g >_L b^t, t\in T \}$ for certain $A^t\in\mathbb{M}_{n,n}$, $b^t\in\mathbb{R}^{n}$ and $T$ an arbitrary index set. In particular, if $\desir \subset \mathbb{R}^{n}$ is a convex cone omitting its apex, one can take $b=0_n$ in Theorem \ref{thm:lexicosep} and write $\desir = \{g\in\mathbb{R}^{n} : A^t g >_L 0_n, t\in T \}$ for certain $A^t\in\mathbb{M}_{n,n}$ (even in $\mathbb{O}_{n,n}$) and $T$ an arbitrary index set. 

At this point, we recall that in $\mathbb{R}^{n}$ there exist maximal convex cones excluding their vertices which are called \emph{semispaces (at the origin)} \citep[see][]{hammer1955maximal}. Thus, a convex set  $\desir\subset\mathbb{R}^{n}$ is a semispace if and only if $0_n \notin \desir$ and for all $g\in\mathbb{R}^{n}\backslash\{0_n\}$, exactly one of $g$ and $-g$ belongs to $\desir$. Furthermore, according to \citet[Lemma 1.1]{S84}, $\desir\subset\mathbb{R}^{n}$ is a semispace if and only if there exists $A\in\mathbb{O}_{n,n}$ (unique, as follows from \citet[p.~139]{MLS88}) such that $\desir = \{g\in\mathbb{R}^{n} : A g >_L 0_n \}$. Thus, every convex cone omitting its apex can be written as an intersection of semispaces.


Concerning the geometry of coherent sets of desirable gambles, any set $\desir \in \mathbb{D}_{n}$ is characterised as a convex cone omitting its apex and containing the set $Q:=\mathbb{R}^{n}_{+} \backslash\{0_n\}$. Thus, as a consequence of the above statement, since any $\desir \in\mathbb{D}_n$ is a convex cone containing $\{e^1,\ldots,e^n\}$, the following proposition follows.
\begin{proposition}
	\label{separdesir}
Let $\desir \in\mathbb{D}_n$ and $\overline{g}\notin \desir$. Then, there exists $A\in\mathbb{O}_{n,n}$ with $A >_L 0_n$ such that $Ag >_L 0_n \geq_L A\overline{g}$ \,for all $g\in\desir$.	
\end{proposition}

\begin{corollary}
	\label{separdesircor}
	For every $\desir \in\mathbb{D}_n$, there exist an index set $T$ and matrices $A^t\in\mathbb{O}_{n,n}$ with $A^t >_L 0_n$ for all $t\in T$ such that $\desir = \{g\in\mathbb{R}^{n} : A^t g >_L 0_n, t\in T \}$.
\end{corollary}
Next we characterise the matrices which are lexicographically greater than $0_n$. We understand that a matrix is unitary if it has ones in the main diagonal. 

\begin{lemma}
	\label{prop:josechar}
	Given $A\in\mathbb{M}_{n,n}$, the following statements are equivalent:
	\begin{itemize}
		\item[$(i)$]   $A >_L 0_n$.
		\item[$(ii)$]  $Ag >_L 0_n$ for all $g > 0_n$.
		\item[$(iii)$] $A=LP$ for some unitary lower-triangular matrix $L$ and some $P\in\mathbb{M}_{n,n}$ such that $p_{\cdot j} > 0_n$ for all $j\in N$. 
	\end{itemize}
\end{lemma}
\begin{proof}
	$(i) \Leftrightarrow (ii)$. If $Ag >_L 0_n$ for all $g > 0_n$, then in particular we have $a_{\cdot j} = A e^j >_L 0_n$ for all $j\in N$ since $e^j > 0_n$, and that is the definition of $A >_L 0_n$. Conversely, assume that $A >_L 0_n$ and so, $A e^j >_L 0_n$ for all $j\in N$. Since any $g=(g_1,\ldots,g_n)>0_n$ can be written as $g=\sum_{i\in N}{g_i e^i}$ with $g_i\geq 0 $ for all $i\in N$ and there is at least one index $j$ such that $g_j$ is strictly positive, then $Ag = \sum_{i\in N}{g_i A e^i} >_L 0_n$. 
	
	$(i) \Leftrightarrow (iii)$. Observe that $A >_L 0_n$ if and only if $A \geq_L 0_n$	and $a_{\cdot j} \neq 0_n$ for each $j\in N$. According to \citet[Proposition 2]{JEML84}, $A \geq_L 0_n$ if and only if $A=LP$ for some unitary lower-triangular matrix $L\in\mathbb{M}_{n,n}$ and some $P \in \mathbb{M}_{n,n}$ such that $p_{ij} \geq 0$ for all $i,j \in N$. Since $a_{\cdot j} = L(p_{\cdot j})$ and $L$ is a regular lower-triangular matrix, then $a_{\cdot j} = 0_n$ if and only if $p_{\cdot j} = 0_n$. Thus, the conclusion follows.
\end{proof}

We say that a coherent set of desirable gambles $ \desir \in \mathbb{D}_n$ is \emph{maximal} if there is no other element $\desir' \in \mathbb{D}_n$ such that $\desir \subset \desir'$.  Thus, we have that the maximal elements in $\mathbb{D}_n$ are the semispaces (at the origin) given by matrices $A\in\mathbb{O}_{n,n}$ satisfying $A>_L 0_n$.  Hence, if we denote by $\M(\mathbb{D}_n)$ the set of all maximal elements in $\mathbb{D}_n$, given $\desir\in\mathbb{D}_n$ one has
\begin{equation}
\desir \in \M(\mathbb{D}_n) \ \Longleftrightarrow \ \exists\, A\in\mathbb{O}_{n,n}, A >_L 0_n \text{ (unique) such that } \desir = \{g\in\mathbb{R}^{n} : Ag >_L 0_n \}.
\label{max:desir}
\end{equation}
This means that there is a one-to-one correspondence between maximal coherent sets of desirables gambles and 
orthonormal matrices whose columns are lexicographically positive. 
Furthermore, as a consequence of Proposition \ref{separdesir}, for any $\desir\in\mathbb{D}_n $ one can write
\begin{equation}
\label{eq:max}
\desir = \bigcap \{ \desir' \in \M (\mathbb{D}_n) : \desir \subset \desir'\},
\end{equation}
recovering thus the characterisation given in \citet[Theorem 21]{couso2011sets}. The above equality and the one in \eqref{max:desir} imply a reformulation of Proposition \ref{separdesir}: if $\desir \in \mathbb{D}_n$ and $g \notin \desir$, then there exists $\desir' \in \M (\mathbb{D}_n)$ such that $\desir \subset \desir'$ and $g \notin \desir'$.\smallskip

The following notions will be useful in the sequel. 

\begin{definition}
	We say that $\mathcal{A}\subset\mathbb{M}_{n,n}$ is   \emph{$L$-convex} if $ \mathcal{A} = \{ A\in\mathbb{M}_{n,n} : A g^t >_L b^t, t\in T  \}  $ 
	for certain vectors $g^t,b^t\in\mathbb{R}^{n}$ for all $t\in T$. In other words, $\mathcal{A}\subset\mathbb{M}_{n,n}$ is $L$-convex if and only if for every $\overline{A}\notin \mathcal{A}$ there exist $g,b\in\mathbb{R}^{n}$ such that $A g >_L b \geq_L \overline{A} g$ for all $A\in\mathcal{A}$.
	
	Analogously, we say that $\mathcal{A}\subset\mathbb{M}_{n,n}$ is an \emph{$L$-convex cone} (omitting its apex) if $ \mathcal{A} = \{ A\in\mathbb{M}_{n,n} : A g^t >_L 0_n, t\in T  \}  $ for certain $g^t\in\mathbb{R}^{n}$ for all $t\in T$. For any $\mathcal{A}\subset\mathbb{M}_{n,n}$, we define the set $\operatorname{Lposi}(\mathcal{A}) := \{ B\in\mathbb{M}_{n,n} : B g >_L 0_n \text{ for any }g\in\mathbb{R}^{n} \text{ satisfying }A g >_L 0_n \text{ for all }A\in\mathcal{A} \}$. Thus, $B\notin \operatorname{Lposi}(\mathcal{A})$ if and only if there is $g\in\mathbb{R}^{n}$ such that $A g >_L 0_n \geq_L B g$\, for all $A\in\mathcal{A}$. 	
\end{definition}

Next we define a new polarity operator which is suitable for general convex cones in $\mathbb{R}^{n}$.



\begin{definition}
	
	For a set $\mathcal{K}\subset\mathbb{R}^{n}$, we define $ \mathcal{K}^{\blacklozenge} := \{ A\in\mathbb{M}_{n,n} : Ag >_L 0_n \ \text{ for all }g \in K\}  $. Furthermore, for a set $\mathcal{A}\subset\mathbb{M}_{n,n}$ we also define $ \mathcal{A}^{\lozenge} := \{g\in\mathbb{R}^{n} : Ag >_L 0_n \ \text{ for all }A\in \mathcal{A}\} $.
\end{definition}

The following facts can be derived from these definitions: 
\begin{enumerate}
	\item $\mathcal{A}^{\lozenge}$ is a convex cone omitting its apex in $\mathbb{R}^{n}$. Moreover, $\mathcal{A} = (\mathcal{A}^{\lozenge})^{\blacklozenge}$ if and only if $\mathcal{A}$ is an $L$-convex cone omitting its apex in $\mathbb{M}_{n,n}$. 
	\item $\mathcal{K}^{\blacklozenge}$ is an $L$-convex cone omitting its apex in $\mathbb{M}_{n,n}$. Moreover, $\mathcal{K} = (\mathcal{K}^{\blacklozenge})^{\lozenge}$ if and only if $\mathcal{K}$ is a convex cone omitting its apex in $\mathbb{R}^{n}$. In particular, this equality holds whenever $\mathcal{K}\in\mathbb{D}_n$.
	\item For any $\mathcal{K},\mathcal{H}\subset\mathbb{R}^{n}$, if $\mathcal{K}\subset \mathcal{H}$ then $\mathcal{H}^{\blacklozenge} \subset \mathcal{K}^{\blacklozenge}$. Analogously, for any $\mathcal{A},\mathcal{B} \subset \mathbb{M}_{n,n}$, if $\mathcal{A}\subset \mathcal{B}$ then $\mathcal{B}^{\lozenge} \subset \mathcal{A}^{\lozenge}$.
	\item $\mathcal{K}^{\blacklozenge} = \{ A\in\mathbb{M}_{n,n} : \mathcal{K}\subset A^{\lozenge}\}$ and $\mathcal{A}^{\lozenge} = \{ g\in\mathbb{R}^{n} : \mathcal{A}\subset g^{\blacklozenge}\}$. 
\end{enumerate}

\begin{proposition}
	\label{proppolard}
	The following statements hold: 
	\begin{itemize}
		\item[$(i)$] If $\mathcal{A} = \{ A\in\mathbb{M}_{n,n} : A g^t >_L 0, t\in T \}$, then $\mathcal{A}^{\lozenge} = \posi\{g^t, t\in T\}$. 
		\item[$(ii)$]   If $\desir = \{g\in\mathbb{R}^{n} : A^t g >_L 0, t\in T \}$, then $\desir^{\blacklozenge} = \operatorname{Lposi}\{A^t, t\in T\}$. 	
	\end{itemize}
\end{proposition} 
\begin{proof}
	$(i)$ Clearly, $g^t \in \mathcal{A}^{\lozenge}$ for all $t\in T$. Since $\mathcal{A}^{\lozenge}$ is a convex cone omitting its apex, then $\posi\{g^t, t\in T\} \subset \mathcal{A}^{\lozenge}$. To prove the converse statement, assume that there is $\overline{g}\in \mathcal{A}^{\lozenge}$ such that $\overline{g} \notin \posi\{g^t, t\in T\}$. By the separation theorem, there exists $A\in\mathbb{M}_{n,n}$ such that $Ag >_L 0_n \geq_L A\overline{g}$\, for all $g\in \posi\{g^t, t\in T\}$. In particular, $A g^t >_L 0_n $ for all $t\in T$, which implies that $A\in\mathcal{A}$. Thus, as $\overline{g}\in \mathcal{A}^{\lozenge}$, one has $ A \overline{g} >_L 0_n$, which entails a contradiction. The proof of $(ii)$ follows the same reasoning as for $(i)$. 
\end{proof}

\vspace{-0.5cm}

\begin{remark}
As a consequence of the above result, if we consider the sets $\mathcal{H}:=\{g\in\mathbb{R}^{n} : g > 0_n\}$ and $\mathcal{B} := \{A\in\mathbb{M}_{n,n} : A >_L 0_n\}$, then one has  $	\mathcal{H}^{\blacklozenge}   =   \mathcal{B} $ and $ \mathcal{B}^{\lozenge}   =  \mathcal{H}$. 
\label{rmk1}
\end{remark}

As this point, we establish an important correspondence between orthonormal matrices with lexicographically positive columns and equivalence classes of full-rank stochastic matrices. Next result guarantees the existence of a full-rank stochastic matrix determining the same semispace as a given orthonormal matrix $A >_L 0_n$, and the proof provides a method for obtaining such a matrix. 


\begin{proposition}
	\label{lem:procedure}
	Let $A\in\mathbb{O}_{n,n}$ be such that $A >_L 0_n$. Then, there exists a full-rank stochastic matrix $P\in\mathbb{T}_{n,n}$ such that $P^{\lozenge} = A^{\lozenge}$. 
\end{proposition}

\begin{proof}
In virtue of Lemma \ref{prop:josechar}, one can write $A=LQ$ with $L$ a unitary lower-triangular matrix and $Q$ such that $q_{\cdot j} > 0_n$ for all $j\in N$. Thus, one has $a_{1\cdot} = q_{1\cdot}$ and $a_{i\cdot} = \sum_{j=1}^{i-1}{l_{ij}q_{j\cdot}} + q_{i\cdot}$ for $i\in N\backslash\{1\}$. Since $A$ is orthonormal, then it follows that $q_{i\cdot} >0_n$ for all $i\in N$, that is, $Q$ does not have null rows, and clearly $Q$ is full-rank as $A$ is. By normalising each row so as that each row becomes a probability mass function, that is, by dividing each row by its sum, one gets the existence of a $P\in\mathbb{T}_{n,n}$. Finally, we observe that $A^{\lozenge} = Q^{\lozenge} = P^{\lozenge}$. 
\end{proof}

\vspace{-0.3cm}

The following proposition studies the way of getting an orthonormal matrix being lexicographically greater than $0_n$ from a full-rank stochastic one. 

\begin{proposition}
\label{prop:gs}
Let $P\in\mathbb{T}_{n,n}$ be a full-rank stochastic matrix. Then, there exists $A\in\mathbb{O}_{n,n}$ with $A >_L 0_n$ such that $A^{\lozenge} = P^{\lozenge}$. 
\end{proposition}

\begin{proof}
We shall denote by $\GS(P)$ the orthogonal matrix obtained from the full-rank stochastic matrix $P \in \mathbb{T}_{n,n}$ by applying the Gram--Schmidt orthogonalisation procedure according to the row order. Let $A\in\mathbb{O}_{n,n}$ be the orthonormal matrix obtained from $\GS(P)$ by normalising each row. Since $P$ have neither null rows nor null columns, it  follows that $\GS(P) >_L 0_n$ and so, $A >_L 0_n$. Finally, the Gram--Schmidt procedure guarantees that $A^{\lozenge} = P^{\lozenge}$.	
\end{proof}

\vspace{-0.3cm}

The next example illustrates that the matrix whose existence has been guaranteed in the Proposition \ref{lem:procedure} is not necessarily unique. 

\begin{example}
	Let us consider the maximal coherent set of desirable gambles $\desir=\{g \in \reals^3: A g >_L 0_3 \}$, where $A=\begin{bmatrix}
	0 & 1/\sqrt{2} & 1/\sqrt{2}\\
	0 & -1/\sqrt{2} & 1/\sqrt{2}\\
	1 & 0 & 0\\
	\end{bmatrix}$. Since $A >_L 0_3$, following Lemma \ref{prop:josechar} $A$ can be written as  
	\begin{equation*}
	A = \begin{bmatrix}
	1 & 0 & 0 \\
	\tau & 1 & 0 \\
	l_{31} & l_{32} & 1\\
	\end{bmatrix}\begin{bmatrix}
	0 & 1/\sqrt{2} & 1/\sqrt{2}\\
	0 & (-1-\tau)/\sqrt{2} & (1-\tau)/\sqrt{2}\\
	1 & 0 & 0\\
	\end{bmatrix} 
	\end{equation*}	
	for any $\tau \leq -1$, $l_{31}, l_{32} \in \mathbb{R}$. According to Proposition \ref{lem:procedure}, by normalising each row of the second matrix in the right-hand side of the equality above, we get that every matrix
	\begin{equation*}	
	P(\tau)=\begin{bmatrix}
	0 & 1/2 & 1/2\\
	0 & (\tau+1)/2\tau & (\tau-1)/2\tau\\
	1 & 0 & 0\\
	\end{bmatrix},
	\end{equation*}
	with $\tau \leq -1$, is a full-rank stochastic matrix which determines $\desir$. 	Finally, it can be checked that $\GS(P(\tau))=A$ holds for any $\tau\leq -1$ (after normalisation). 
\end{example}

The above results suggest the definition of the $\lozenge$-equivalence class of a given matrix $A\in\mathbb{M}_{n,n}$ as the set of matrices having the same polar that $A$, that is, $[A]_{\lozenge} := \{P\in\mathbb{M}_{n,n} : P^{\lozenge} = A^{\lozenge}\}$. According to this definition, we have that there is a one-to-one correspondence between maximal coherent sets of desirable gambles and $\lozenge$-equivalence classes of stochastic matrices of full rank. 
\smallskip

\begin{definition}
	We say that a nonempty subset of $\mathbb{M}_{n,n}$ is an \emph{$L$-credal set} if it is the intersection with $\mathbb{T}_{n,n}$ of some $L$-convex cone in $\mathbb{M}_{n,n}$. We shall denote by $\mathbb{G}_n$ the family of all  $L$-credal sets.
\end{definition}

We are now in position to define the function $\mathbf{G} : \mathbb{D}_n\rightarrow\mathbb{G}_n$ which maps coherent sets of desirable gambles into $L$-credal sets and it is the key for the equivalence of both theories. For a  coherent set of desirable gambles $\desir\in\mathbb{D}_n$, we associate the $L$-credal set
\begin{equation}
\label{maplc}
\mathbf{G}(\mathcal{K}) := \mathcal{K}^{\blacklozenge} \cap \mathbb{T}_{n,n}.
\end{equation}
We aim at showing that $\mathbf{G}$ is a bijection.

\begin{theorem}
	\label{theo:Lcredal}
	The mapping $\mathbf{G}:\mathbb{D}_n \rightarrow \mathbb{G}_n$ defined in \eqref{maplc} is a bijection whose inverse is given by 
	$\mathbf{G}^{-1}(\mathcal{P}) := \mathcal{P}^{\lozenge},
	$ 
for every $\mathcal{P}\in\mathbb{G}_n$. 
\end{theorem}
\begin{proof} 
From the definition of the $\blacklozenge$-polarity operator, $\mathbf{G}(\desir)$ is an $L$-credal set, for any 
$\desir\in \mathbb{D}_n$. 
As $\mathcal{H}\subset \desir$, then $ \desir^{\blacklozenge} \subset \mathcal{H}^{\blacklozenge} = \mathcal{B} $ (see Remark \ref{rmk1}). One also has $\desir^{\blacklozenge} = \{A\in \mathbb{M}_{n,n} : \desir \subset A^{\lozenge}\}$. Since $\desir$ is determined by orthonormal matrices, then $\desir^{\blacklozenge}$ contains orthonormal matrices with lexicographically positive columns and, as a consequence of Proposition \ref{lem:procedure}, $\desir^{\blacklozenge}$ also contains 
full-rank stochastic matrices, which shows that $\mathbf{G}(\desir)$ is nonempty. Now, if $\mathcal{P}\in\mathbb{G}_n$, one has that $\mathbf{G}^{-1}(\mathcal{P}) = \mathcal{P}^{\lozenge} $ is a convex cone omitting its apex. On the other hand, as $\mathcal{P} \subset \mathbb{T}_{n,n} \subset \mathcal{B}$, then $Q = \mathcal{B}^{\lozenge} \subset \mathcal{P}^{\lozenge} $ and so, $\mathbf{G}^{-1}(\mathcal{P}) \in \mathbb{D}_n$. 

To see that $\mathbf{G}$ is one-to-one, we just need to show $\mathbf{G}(\mathbf{G}^{-1}(\mathcal{P})) = \mathcal{P}$ for any 
$\mathcal{P}\in\mathbb{G}_n$ and also
  $\mathbf{G}^{-1}(\mathbf{G}(\desir)) = \desir$ for $\mathcal{K}\in\mathbb{D}_n$. First, $\mathbf{G}(\mathbf{G}^{-1}(\mathcal{P})) = \mathbf{G}(\mathcal{P}^{\lozenge}) = \mathcal{P}^{\lozenge\blacklozenge} \cap \mathbb{T}_{n,n} = \operatorname{Lposi}(\mathcal{P}) \cap \mathbb{T}_{n,n} = \mathcal{P} $. On the other hand, 
$ \mathbf{G}^{-1}(\mathbf{G}(\desir)) = \mathbf{G}^{-1}(\desir^{\blacklozenge} \cap \mathbb{T}_{n,n}) = (\desir^{\blacklozenge} \cap \mathbb{T}_{n,n})^{\lozenge} = \desir^{\blacklozenge\lozenge} = \desir $ as $\desir$ is a convex cone omitting its apex.  
\end{proof}

\subsection{Closing the circle, or preserving conditioning}

As for almost desirability, one wants to verify that $\mathbf{G}$ is not only a bijection but also an isomorphism. To make sense of this claim, we thus have first to specify which operations and relations we decide to consider (in model-theoretic terms, the signature), 
and how they are defined over sets of gambles and over sets of stochastic matrices (in model-theoretic terms, the interpretation). 
Finally, we have to verify that the map $\mathbf{G}$ preserves the considered operations and relations.
As before, 
here we are only interested in conditioning. 

Without loss of generality we assume that $\Pi \subsetneq \Omega$ has cardinality $m$. In the case of stochastic matrices, conditioning has to be defined by slightly modifying the approach by \cite{blume1991lexicographic}. This is because we want to be sure that the result of the operation is a square stochastic matrix. 
With this aim in mind, we first define the following reduction rule for matrices:  
\begin{description}
	\item[(R)] Given $A \in \mathbb{M}_{n,m}$, for every $i \in N$, discard the $i$-th row $a_{i\cdot}$ whenever it is a linear combination of $a_{1\cdot}, \dots, a_{i-1\cdot}$ (and thus in particular when it is equal to $0_m$).
\end{description}

Let $P' \in \mathbb{M}_{n,m}$ be the matrix obtained by projecting  on $\Pi$ the conditioning $p(\cdot |\Pi)$, or taking $0_m$ when it is undefined, for each row $p$ of $P \in \mathbb{T}_{n,n}$.
Define $P\rfloor_\Pi$ as the matrix obtained from $P'$ by applying rule (R). 
By an immediate  application of properties of minors and cofactors, we  get that $P\rfloor_\Pi \in \mathbb{T}_{m,m}$. Moreover $(P\rfloor_\Pi)\rfloor_\Delta=(P \rfloor_\Delta)$, for $\Delta \subset \Pi$.  
Hence, the following operation is always defined.

\begin{definition}\label{def:conlexico2} 
	Let  $\mathcal{P}\subset \mathbb{T}_{n,n}$, with $n>1$. Its conditioning on $\Pi$ is the set 
	$(\mathcal{P}\rfloor_\Pi):=\{ (P\rfloor_\Pi)  \mid P \in  \mathcal{P}\}\subset \mathbb{T}_{m,m}$.
\end{definition}

From Definition \ref{condition_set}, it is immediate to verify that $(\desir \rfloor_\Pi) \in \mathbb{D}_m$ whenever $\desir \in \mathbb{D}_n$, and that $\mathbb{D}_n$ is closed under conditioning. Moreover, $(\desir \rfloor_\Pi) \in \M(\mathbb{D}_m)$ whenever $\desir \in \M(\mathbb{D}_n)$.
To conclude, we verify that polarity preserves conditioning.

\begin{theorem}\label{prop:conditioningwork}
	Let $\desir \in \mathbb{D}_n$, then $ (\mathbf{G}(\desir)\rfloor_\Pi)= \mathbf{G}(\desir \rfloor_\Pi) \in \mathbb{G}_m$. 
\end{theorem}
\begin{proof} 
	It is enough to prove  the claim  for maximal consistent sets of desirable gambles. 
	Hence, let $\desir \in \M(\mathbb{D}_n)$. 
	We first define a conditioning operation on orthogonal matrices. Let $A \in \mathbb{O}_{n,n}$. Its conditioning on $\Pi$ is the matrix $A\rfloor_\Pi$ obtained by the following procedure: (i)  erase all $k$-th column from $A$,  with $k \in \{m+1, \dots, n\}$; (ii)  apply rule (R) to the matrix obtained after the previous point; (iii)  assume the matrix you obtained after the previous point is $B$. By linear algebra, $B \in \mathbb{U}_{m,m}$. Hence, $A\rfloor_\Pi:=\GS(B)\in \mathbb{O}_{m,m}$.
	Note that the operation also preserves the property of being lexicographic positive for columns. 
	Thus, let $A\in \mathbb{O}_{n,n}$, $A >_L 0_n$, such that $\desir = A^\lozenge$. Both $(\desir\rfloor_\Pi), (A^\lozenge\rfloor_\Pi) \in \M(\mathbb{D}_m)$. This means that, in order to show that $(\desir\rfloor_\Pi)=(A^\lozenge\rfloor_\Pi)$, it is enough to verify one of the two inclusions. So, let 
	$f \in (\desir\rfloor_\Pi)$. By definition $f\lceil_{\Pi^c} \in \desir$, and thus $A(f\lceil_{\Pi^c})>_L0_n$. But this means that $Bf>_L0_n$, since $f\lceil_{\Pi^c}$ agrees on $\Pi$ with $f$, and is $0$ elsewhere. Thence $\GS(B)f>_L0_n$, meaning that $f \in A^\lozenge\rfloor_\Pi$.  
	Now, because of the properties of 
	the procedures given by Propositions \ref{lem:procedure} and \ref{prop:gs},  it holds that   $P \in [A]_\lozenge$ if and only if $P\rfloor_\Pi \in [A\rfloor_\Pi]_\lozenge$, for $P \in \mathbb{T}_{n,n}$.
	Finally, we can apply Theorem \ref{theo:Lcredal} and conclude that $ (\mathbf{G}(\desir)\rfloor_\Pi)= \mathbf{G}(\desir \rfloor_\Pi) $.
\end{proof}


\section{Conclusions}
In this paper we have shown  that (conditional) sets of lexicographic probabilities and (conditional) sets of desirable gambles  are isomorphic structures. In doing so,  we
have  provided a duality transformation (via orthogonal and stochastic matrices) that allows  us to go from a coherent set of desirable gambles to an equivalent (convex) set of lexicographic probabilities and vice versa.  
As future work we plan to complete this analysis by including other operations, such as marginalisation (this should be straightforward), and structural judgements such as independence. It would be also of great interest to study what are the geometric properties of lexicographic convex sets of stochastic matrices, and what happens for gambles on infinite sample spaces.

\acks{The authors are grateful to the referees for their constructive comments and helpful suggestions which have contributed to the final preparation of the paper. J. Vicente-P\'erez was partially supported by MINECO of Spain and ERDF of EU, Grants 	MTM2014-59179-C2-1-P and ECO2016-77200-P.}

%
%


\bibliography{biblio}

\begin{thebibliography}{18}
\providecommand{\natexlab}[1]{#1}
\providecommand{\url}[1]{\texttt{#1}}
\expandafter\ifx\csname urlstyle\endcsname\relax
  \providecommand{\doi}[1]{doi: #1}\else
  \providecommand{\doi}{doi: \begingroup \urlstyle{rm}\Url}\fi

\bibitem[Blume et~al.(1991)Blume, Brandenburger, and
  Dekel]{blume1991lexicographic}
L.~Blume, A.~Brandenburger, and E.~Dekel.
\newblock Lexicographic probabilities and choice under uncertainty.
\newblock \emph{Econometrica}, 59\penalty0 (1):\penalty0 61--79, 1991.

\bibitem[Couso and Moral(2011)]{couso2011sets}
I.~Couso and S.~Moral.
\newblock Sets of desirable gambles: conditioning, representation, and precise
  probabilities.
\newblock \emph{International Journal of Approximate Reasoning}, 52\penalty0
  (7):\penalty0 1034--1055, 2011.

\bibitem[Cozman(2015)]{cozman2015}
F.~G. Cozman.
\newblock Some remarks on sets of lexicographic probabilities and sets of
  desirable gambles.
\newblock In \emph{9th ISIPTA}, Pescara, Italy, 2015.

\bibitem[De~Cooman and Quaeghebeur(2012)]{de2012exchangeability}
G.~De~Cooman and E.~Quaeghebeur.
\newblock Exchangeability and sets of desirable gambles.
\newblock \emph{International Journal of Approximate Reasoning}, 53\penalty0
  (3):\penalty0 363--395, 2012.

\bibitem[de~Finetti(1937)]{finetti1937}
B.~de~Finetti.
\newblock La pr\'evision: ses lois logiques, ses sources subjectives.
\newblock \emph{Annales de l'Institut Henri Poincar\'e}, 7:\penalty0 1--68,
  1937.

\bibitem[Hammer(1955)]{hammer1955maximal}
P.~C. Hammer.
\newblock Maximal convex sets.
\newblock \emph{Duke Math. J}, 22:\penalty0 103--106, 1955.

\bibitem[Hodges(1997)]{hodges1997shorter}
W.~Hodges.
\newblock \emph{A shorter model theory}.
\newblock Cambridge University Press, 1997.

\bibitem[Mart{\'\i}nez-Legaz(1983)]{martinez1983exact}
J.~Mart{\'\i}nez-Legaz.
\newblock Exact quasiconvex conjugation.
\newblock \emph{Zeitschrift f{\"u}r Operations-Research}, 27\penalty0
  (1):\penalty0 257--266, 1983.

\bibitem[Mart{\'\i}nez-Legaz(1984)]{JEML84}
J.~E. Mart{\'\i}nez-Legaz.
\newblock Lexicographical order, inequality systems and optimization.
\newblock In \emph{Proceedings of the 11th IFIP Conference on System Modelling
  and Optimization}, volume~59 of \emph{Lecture Notes in Control and Inform.
  Sci.}, pages 203--212. Springer, 1984.

\bibitem[Mart{\'\i}nez-Legaz and Singer(1988)]{MLS88}
J.~E. Mart{\'\i}nez-Legaz and I.~Singer.
\newblock The structure of hemispaces in $\mathbb{R}^{n}$.
\newblock \emph{Linear Algebra and its Applications}, 110:\penalty0 117--179,
  1988.

\bibitem[Miranda and Zaffalon(2010)]{zaffalon2010e}
E.~Miranda and M.~Zaffalon.
\newblock Notes on desirability and conditional lower previsions.
\newblock \emph{Annals of Mathematics and Artificial Intelligence}, 60\penalty0
  (3-4):\penalty0 251--309, 2010.

\bibitem[Rockafellar(1970)]{R70}
R.~T. Rockafellar.
\newblock \emph{Convex Analysis}.
\newblock Princeton University Press, 1970.

\bibitem[Seidenfeld(2000)]{seidenfeld2000remarks}
T.~Seidenfeld.
\newblock Remarks on the theory of conditional probability: Some issues of
  finite versus countable additivity.
\newblock 2000.

\bibitem[Seidenfeld et~al.(1990)Seidenfeld, Schervish, and
  Kadane]{seidenfeld1990decisions}
T.~Seidenfeld, M.~J. Schervish, and J.~B. Kadane.
\newblock Decisions without ordering.
\newblock In \emph{Acting and reflecting}, pages 143--170. Springer, 1990.

\bibitem[Singer(1984)]{S84}
I.~Singer.
\newblock Generalized convexity, functional hulls and applications to conjugate
  duality in optimization.
\newblock In \emph{Selected topics in operations research and mathematical
  economics}, volume 226 of \emph{Lecture Notes in Econ. and Math. Systems},
  pages 49--79. Springer, 1984.

\bibitem[Van~Camp et~al.(2017)Van~Camp, Miranda, and De~Cooman]{van2017lexico}
A.~Van~Camp, E.~Miranda, and G.~De~Cooman.
\newblock Lexicographic choice functions without archimedeanicity.
\newblock In \emph{Soft Methods for Data Science}, pages 479--486. Springer,
  2017.

\bibitem[Walley(1991)]{Walley91}
P.~Walley.
\newblock \emph{Statistical Reasoning with Imprecise Probabilities}.
\newblock Chapman \& Hall/CRC Monographs on Statistics \& Applied Probability.
  1991.

\bibitem[Williams(1975)]{williams1975}
P.~M. Williams.
\newblock Notes on conditional previsions.
\newblock Technical report, University of Sussex, 1975.

\end{thebibliography}

\end{document}